\documentclass[11pt]{article}
\usepackage{amsmath,mathrsfs}
\usepackage{amssymb}
\usepackage{amsthm}
\usepackage{indentfirst}
\usepackage{color}
\allowdisplaybreaks

\newtheorem{theorem}{Theorem}[section]
\newtheorem*{theorem*}{Theorem}
\newtheorem{assumption}{Assumption}
\newtheorem{corollary}[theorem]{Corollary}
\newtheorem{remark}[theorem]{Remark}
\newtheorem{lemma}[theorem]{Lemma}
\newtheorem{definition}[theorem]{Definition}

\newtheorem{proposition}[theorem]{Proposition}
\newtheorem{example}[theorem]{Example}
\newtheorem*{proposition*}{Proposition}

\newcommand{\R}{\mathbb{R}}

\newcommand{\be}{\begin{eqnarray*}}
\newcommand{\ee}{\end{eqnarray*}}
\newcommand{\ba}{\begin{align*}}
\newcommand{\bpm}{\begin{pmatrix}}
\newcommand{\epm}{\end{pmatrix}}

\begin{document}

\title{Monogenic functions over real alternative $\ast$-algebras: \\ the several hypercomplex variables   case}
\author
{Zhenghua Xu$^1$\thanks{This work was partially supported by  Anhui Provincial Natural Science Foundation (No. 2308085MA04), Fundamental Research Funds for the Central Universities (No. JZ2025HGTG0250) and China Scholarship Council (No. 202506690052).}
,\ Chao Ding$^2$\thanks{This work was partially supported by the National Natural Science Foundation of China (No.  12271001).}
,\ and Haiyan Wang$^3$
\\
\\
\emph{$^1$\small School of Mathematics, Hefei University of Technology,}\\
\emph{\small  Hefei, 230601, P. R. China}\\
\emph{\small zhxu@hfut.edu.cn}
\\
\emph{\small $^2$Center for Pure Mathematics, School of Mathematical Sciences, Anhui University, }\\
\emph{\small Hefei 230601, P. R. China} \\
\emph{\small E-mail address:   cding@ahu.edu.cn}
\\
\emph{\small $^3$School of Science,
 Tianjin University of Technology and Education, }\\
\emph{\small Tianjin 300222, P. R. China} \\
\emph{\small E-mail address:    whaiyan@mail.ustc.edu.cn}
 }


\maketitle

\begin{abstract}
The notion of monogenic (or regular) functions, which is a correspondence of holomorphic functions, has been studied extensively in hypercomplex analysis, including  quaternionic, octonionic, and Clifford analysis.  Recently, the concept of  monogenic  functions over real alternative $\ast$-algebras has been introduced to unify several   classical monogenic functions theories. In this paper,  we initiate the study of monogenic functions of  several hypercomplex variables over real alternative $\ast$-algebras, which naturally extends the theory of several complex variables to a very general setting. In this new setting,    we develop some  fundamental properties, such as
Bochner-Martinelli formula,  Plemelj-Sokhotski formula, and Hartogs extension theorem.

 \end{abstract}
{\bf Keywords:}\quad Functions  of  several  hypercomplex variables;  alternative algebras;\\ Bochner-Martinelli formula;  Plemelj-Sokhotski formula; Hartogs theorem. \\
{\bf MSC (2020):}\quad  Primary: 30G35;  Secondary:  17D05, 32A26,  47G10.
\section{Introduction}
This paper shall  investigate  the notion of  monogenic functions of  several hypercomplex variables over real alternative $\ast$-algebras, which   extends   holomorphic functions theory in several complex variables to  the setting of non-commutative and  non-associative algebras.

For holomorphic functions of several complex variables, the  integral representation formula where the integration is carried out over the
whole boundary of the domain   was obtained firstly by Enzo Martinelli in 1938   \cite{Martinelli-38} and Salomon Bochner in 1943 \cite{Bochner-43},   independently and by different methods. Now there are a large number of literatures devoted to the Bochner-Martinelli type theorems and  applications. Without claiming completeness, we refer the reader to    \cite{Kytmanov-95,Range-86,RSS-02} and references therein for more details of the Bochner-Martinelli  theorems for (holomorphic) functions in several complex variables.
In 1957,   Lu   and Zhong \cite{Lu-Zhong-57} initiated  the study of the boundary behaviour of the Bochner-Martinelli  integral  on bounded domains with $C^{2}$  smooth boundary in  several complex variables. Thereafter, investigations on the  singular integral, such as Plemelj-Sokhotski
formula, and  higher order versions related to the Bochner-Martinelli  integral become more active; see e.g.  \cite{Lin-88,Lin-Qiu-88,Lin-Qiu-04,Qian-Zh-01,Range-86}.

As a generalization of  holomorphic functions theory of one complex variable, in the 1930s  the theory of \textit{regular functions}  was initially introduced in quaternionic analysis and  later extended to   \textit{monogenic functions} in Clifford analysis. From 1980s to now, monogenic (or regular) functions had been studied comprehensively, cf. \cite{Brackx,Colombo-Sabadini-Sommen-Struppa-04,Delanghe-Sommen-Soucek-92,Gurlebeck}.
The Bochner-Martinelli theorem had also been    discussed for some variants of monogenic functions, such as   in  two  higher dimensional variables  over real Clifford algebras \cite{Brackx-Pincket-84} and  complex Clifford algebras \cite{Sommen-87}.  A canonical relation was  shown that  the   Cauchy integral formula  for    monogenic  functions in various setting  may  imply the   Bochner-Martinelli integral formula  for holomorphic functions of several complex variables; see  e.g. \cite{BDDS-09,Bory-18,Li-Ren-22}.

The  $k$-Cauchy-Fueter operator  is a generalization of the
 Cauchy-Fueter operator in quaternionic analysis. In   several quaternionic variables case,    the Bochner-Martinelli integral formula associated to $k$-Cauchy-Fueter operator   had   been discussed  \cite{Wang-10,Wang-Ren-14-JMP}. In 2020,
 the Bochner-Martinelli type formula for   tangential $k$-Cauchy-Fueter operator has also been established \cite{Ren-Shi-Wang-20}.

 As is known,   quaternionic analysis and Clifford analysis deal both  with the case that the involved  algebras are associative. As another  generalization of holomorphic functions of one complex variable,  Dentoni and Sce in 1973 studied regular functions over the non-associative   algebra of octonions \cite{Dentoni-Sce}. These theories  had always been   developing   in parallel, until  the notion of  monogenic (or regular) functions over  alternative algebras appears; see \cite{Perotti-22} and \cite{Ghiloni-Stoppato-24-2}. In \cite{HRX}, the authors  initiated the study of  monogenic functions  of one   hypercomplex variable   over real alternative $\ast$-algebras and developed    fundamental properties, focusing on the Cauchy-Pompeiu integral formula and
Taylor series expansion. In this paper, we    concentrate on the  several hypercomplex variables  case,  which naturally extends the theory of several complex variables to a very general setting and unifies complex, quaternionic,  octonionic, and Clifford analysis.   In particular, in this new setting we establish  some  fundamental results, such as Bochner-Martinelli formula,  Plemelj-Sokhotski formula, and Hartogs extension theorem, which are   cornerstones of   holomorphic functions theory in several complex variables.

The remaining part of  this paper is organized as follows.  In Section 2, we review   basic properties of real alternative $\ast$-algebras and then introduce  monogenic functions of  several hypercomplex variables. Over real alternative $\ast$-algebras,  Section 3 is dedicated to   the Bochner-Martinelli integral formula; see Theorem \ref{Bochner-Martinelli-theorem}, which  reduces to be  the Cauchy-Pompeiu   integral  formula in one hypercomplex variable. In Section 4, we  shall study  the boundary behaviour of the Bochner-Martinelli integral by establishing   the Plemelj-Sokhotski formula in Theorem \ref{Plemelj-Sokhotski-theorem}, which is a first result in the several hypercomplex variables case.
Finally, in Section 5, we discuss the solvability of the inhomogeneous Cauchy-Riemann equation with compactly supported inhomogeneous terms, which allows to prove  a Hartogs extension theorem  in Theorem \ref{Hartogs} for monogenic functions  over  real alternative $\ast$-algebras.
 \section{Preliminaries}
In this section,  for the convenience of the reader, we first recall  some preliminary results for alternative algebras from  \cite{Okubo,Schafer} and  hypercomplex subspaces from \cite{Ghiloni-Stoppato-24-2,Perotti-22}.  With this in hand,  we  can introduce   the notion  of  monogenic   functions of several hypercomplex variables over real alternative    $\ast$-algebras.

\subsection{ Real alternative $\ast$-algebras and hypercomplex subspaces}
\noindent Let $\mathbb{A}$ be a real alternative algebra with a unity.  All  alternative algebras   obey
 the Artin theorem saying  \textit{the subalgebra generated by two elements of $\mathbb  A$ is associative}.
\begin{definition}
A real algebra   $\mathbb  A$    is called to be $\ast$-algebra if $\mathbb  A$ is equipped with an anti-involution, which is a real linear map $^{c}:\mathbb  A \rightarrow \mathbb  A$,  $a\mapsto a^{c}$ satisfying that $a^{c}= a$ for $a\in \mathbb{R}=\mathbb{R} 1\subseteq \mathbb  A $ and
$$ (a^{c})^{c} = a,   \ (ab)^{c} =b^{c}a^{c},  \quad a,b \in\mathbb A.$$
\end{definition}

\begin{assumption}
Assume that $(\mathbb A, +,  \cdot, ^{c})$ is a real alternative $\ast$-algebra with a unity,  of finite dimension $d>1$ as   a real vector space, and  equipped with an anti-involution $^{c}$.    Additionally, we endow $\mathbb A$   with the natural topology and differential structure as a real vector space.
\end{assumption}

It is known that the set of such real alternative $\ast$-algebras is very large, which can include  division algebras, Clifford algebras, and  dual quaternions.
\begin{example} [Division algebras]
The classical division algebras of  the complex numbers $\mathbb{C}$, quaternions $\mathbb{H}$ and  octonions $\mathbb{O}$ are real  $\ast$-algebras, where $\ast$-involutions are  the standard conjugations of  complex numbers, quaternions, and  octonions, respectively.
\end{example}
\begin{example} [Clifford algebras]
 The   Clifford algebra   $\mathbb{R}_{0,m}$ is an associative $\ast$-algebra,  which is generated  by the standard orthonormal basis $\{e_1,e_2,\ldots, e_m\}$ of the $m$-dimensional real Euclidean space  $\mathbb{R}^m$ by assuming   $$e_i e_j + e_j e_i= -2\delta_{ij}, \quad 1\leq i,j\leq m,$$
where    $\ast$-involution is given by the Clifford conjugation.
\end{example}

For each element $x$ in the $\ast$-algebra ${\mathbb A}$, its \textit{trace}  is  given by $t(x):= x+x^{c}\in {\mathbb A}$ and its   (squared) \textit{norm} is $n(x):= xx^{c}\in {\mathbb A}.$ Let
 $$\mathbb{S}_{{\mathbb A}}:= \{x  \in \mathbb A : t(x)=0,  n(x)=1  \}$$
 be the  sphere  of the imaginary units of $\mathbb A$ compatible with the  $\ast$-algebra structure of $\mathbb A$.

For  $J \in \mathbb{S}_{{\mathbb A}}$,   denote by $\mathbb{C}_{J}:=\langle1,J\rangle \cong \mathbb{C},$ the subalgebra of ${\mathbb A}$ generated by $1$ and $J$.
 The quadratic cone of ${\mathbb A}$ is defined  as
$$Q_{{\mathbb A}}:= \mathbb{R} \cup \big\{x \in {\mathbb A} \mid t(x) \in  \mathbb{R}, \ n(x)\in \mathbb{R}, \  4\,  n(x)>t(x)^{2} \big\}.$$
It is known that
$$  Q_{\mathbb A}=\bigcup_{J\in\mathbb S_{\mathbb A}} \mathbb C_J,$$
 and
 $$\mathbb C_I \cap \mathbb C_J=\mathbb R, \qquad \forall\ I, J\in\mathbb S_{\mathbb A}, \ I\neq \pm J.$$
Interested readers can refer to   \cite{Ghiloni-Perotti} for more  properties of the quadratic cone of ${\mathbb A}$.
\begin{assumption}
Assume  $\mathbb{S}_{{\mathbb A}}\neq \emptyset$.
\end{assumption}



\begin{definition}
Let $M$ be a real vector subspace of our $\ast$-algebra $\mathbb A$. An ordered real vector
basis $(v_0,v_1,  \ldots, v_m)$  of $M $ is called a hypercomplex basis of $M$ if: $m \geq 1$; $v_0 =1$; $v_s \in \mathbb{S}_{{\mathbb A}}$
and $v_sv_t = -v_tv_s$ for all distinct $s, t\in \{1,  \ldots, m\}$. The subspace $M$ is called a hypercomplex
subspace of $\mathbb A$ if $\mathbb{R}  \subsetneq   M  \subseteq Q_{{\mathbb A}}$.
\end{definition}

Equivalently, a basis $(v_0=1,v_1,  \ldots, v_m)$ is a hypercomplex basis if, and only if, $t(v_s)=0, n(v_s)=1$
and $t(v_s v_t^{c})=0$ for all distinct $s, t\in \{1,\ldots , m\}$.

\begin{assumption} Within the alternative $\ast$-algebra $\mathbb A$, we fix a hypercomplex subspace $M$, with a fixed hypercomplex basis  $\mathcal{B}=(v_0,v_1,  \ldots, v_m)$. Moreover, $\mathcal{B}$ is completed to a real vector basis $\mathcal{B}' =(v_0, v_1,  \ldots , v_d)$ of  $\mathbb A$ and  $\mathbb A$ is endowed with the standard Euclidean scalar product $\langle\cdot,\cdot\rangle$ and norm $ | \cdot |$ associated to $\mathcal{B}'$. If this is the case, then,   for all $x, y\in M$,
$$t(xy^{c})=t(y^{c}x)=2 \langle x,y\rangle,$$
$$n(x)=n(x^c)=|x|^{2}.$$
 \end{assumption}

The following example allows to  unify several  theories  of monogenic functions in quaternionic, octonionic, and Clifford analysis.

\begin{example}\label{example:strong-regular}  The best-known examples of    hypercomplex subspaces   in real alternative $\ast$-algebras   are given by
\begin{eqnarray*}
M=
\begin{cases}
\mathbb H \ or \  \mathbb H_r,  \quad  \quad \ \mathbb A=\mathbb H, \qquad   \   Q_{\mathbb A}=\mathbb H,
\\
\mathbb O, \qquad  \quad \qquad  \, \mathbb A=\mathbb O, \qquad   \  Q_{\mathbb A}=\mathbb O,
\\
   \mathbb{R}^{m+1}, \qquad  \quad \ \ \mathbb A=\mathbb{R}_{0, m}, \ \ \  Q_{\mathbb A}\supseteq\mathbb R^{m+1},
\end{cases}
\end{eqnarray*}
where $\mathbb H_r$ denotes the three-dimensional subspace  of reduced quaternions
$$\mathbb H_r=\{x=x_0+x_1i+x_2j\in \mathbb{H}: x_0,x_1,x_2\in \mathbb{R}\},$$
and $\mathbb R^{m+1}$ denotes the space of paravectors in $\mathbb{R}_{0, m}$.
\end{example}

Two immediate consequences of the Artin  theorem follow, which are  useful in our subsequent proof.
\begin{proposition}\label{real}
For any $r\in \mathbb{R}$ and $x, y \in \mathbb{A},$ it holds that
$$[r,x,y]=0.$$
\end{proposition}
\begin{proposition}\label{artin}
For any $x\in M$ and $y \in \mathbb{A},$ it holds that
$$[x,x,y]=[x^{c},x,y]=0.$$
\end{proposition}
At the end of this subsection, we recall a result from \cite[Remark 2.27]{Ghiloni-Stoppato-24-2}.
\begin{proposition}\label{norm}
 There exists some $C\geq1$ such that
$$|xy|\leq C | x |  | y |, \quad x\in M,y \in \mathbb{A}.$$
\end{proposition}

\subsection{Monogenic functions of several hypercomplex variables}

Throughout this paper,  an element   $x=(x_1,\ldots,x_n)\in (\R^{m+1})^{n}=\R^{(m+1)n}$ shall be identified with the element in $M^{n}$, via
$$ x_{j}= \sum_{s=0}^{m}x_{j,s}v_s =\sum_{s=0}^{m}v_s x_{j,s} \in M,  \quad j=1,\ldots,n, $$
with $x_{j,s}\in \R$.

 We consider the functions  $f:\Omega\longrightarrow  \mathbb{A}$, where $\Omega\subseteq \R^{(m+1)n}$ is a domain (i.e., connected and open set). As usual, denote by $\mathbb{N}=\{0,1,2,\ldots\}$ all non-negative integers.
 For $k\in \mathbb{N}\cup \{\infty\}$,    denote by $C^{k}(\Omega,\mathbb{A})$ the set of all functions  $f (x)=\sum_{s=0}^{d}  v_{s}   f_{s}(x)$ with real-valued $f_{s}(x)\in  C^{k}(\Omega)$. In this setting, the global Dirac operator is given by
$$\overline{\partial}=(\overline{\partial}_{1},\ldots,\overline{\partial}_{n}),$$
where the Dirac operator $\overline{\partial}_{j}= \overline{\partial}_{x_j}$ corresponding to the $j$-th copy of $M$ is
$$\overline{\partial}_{j}=\overline{ \partial}_{x_{j}}= \sum_{s=0}^{m}v_s\partial_{x_{j,s}},  \quad j=1,\ldots,n. $$

Now  we  give the definition of  monogenic functions in  several hypercomplex variables over real alternative $\ast$-algebras.

\begin{definition}[Monogenic  functions] \label{monogenic-Clifford}
Let $n\in \mathbb{N}\setminus\{0\}$ and   let  $f\in C^{1}(\Omega,\mathbb{A})$, where $\Omega $ is  a domain   in $M^{n}$.  The function $f (x)=\sum_{s=0}^{d}  v_{s}   f_{s}(x)$ is called  left monogenic  with respect to $\mathcal{B}$ in $\Omega $ if
$$\overline{\partial} f(x)=(\overline{\partial}_{ 1}f(x),\ldots,\overline{\partial}_{ n}f(x))=0, \quad x\in \Omega,$$
where
$$ \overline{\partial}_{ j}f(x)  =\sum _{s=0}^{m}v_{s} \frac{\partial f}{\partial x_{j, s}}(x)
= \sum _{s =0}^{m} \sum _{t=0}^{d}   v_{s}v_{t} \frac{\partial f_{t}}{\partial x_{j,s}}(x)=0, \quad j=1,\ldots,n. $$
Similarly, the function $f$ is called  right  monogenic  with respect to $\mathcal{B}$  in $\Omega $ if
$$f(x)\overline{\partial} =(f(x)\overline{\partial}_{ 1},\ldots,f(x)\overline{\partial}_{ n})=0, \quad x\in \Omega,$$
where
$$f(x) \overline{\partial}_{ j} =\sum _{s=0}^{m}\frac{\partial f}{\partial x_{j, s}}(x)v_{s}
= \sum _{s =0}^{m} \sum _{t=0}^{d}   \frac{\partial f_{t}}{\partial x_{j,s}}(x)v_{t}v_{s}=0, \quad j=1,\ldots,n. $$

\end{definition}
Denote by $\mathcal {M}(\Omega, \mathbb{A})$ or $\mathcal {M}^{L}(\Omega, \mathbb{A})$  (resp. $\mathcal {M}^{R}(\Omega, \mathbb{A})$)   the function class of  all  left  (resp. right) monogenic functions   $f:\Omega \rightarrow \mathbb{A}$.

\begin{remark}\label{C1-C00}{\rm
 When $\mathbb{A}$ is  associative, such as   $\mathbb{A}=\mathbb{R}_{0,m}$,      $\mathcal {M}^{L}(\Omega, \mathbb{A})$ is   a right $\mathbb{A}$-module.
Meanwhile, when  $\mathbb{A}$ is     non-associative, such as $\mathbb{A}=\mathbb{O}$,  $\mathcal{M}^{L}(\Omega, \mathbb{A})$ is generally  not a right   $\mathbb{A}$-module.}
\end{remark}

\begin{remark}\label{harmonic}
{\rm By Theorem \ref{Bochner-Martinelli-theorem} and Remark \ref{kernel-harmonic-remark} below,
all functions  in $\mathcal {M}(\Omega,  \mathbb{A})$ are harmonic and then are of class $C^{\infty}(\Omega,\mathbb{A})$.}  \end{remark}

\subsection{Integrals  on a hypercomplex subspace}
To discuss  the integral representation     over   real alternative $\ast$-algebras in   subsequent sections,  we give some notations and properties of    integrals over   real alternative $\ast$-algebras.
\begin{definition}\label{integral}
Let    $\Omega $ be a   domain in $M^{n}$. For $f=\sum_{t=0}^{d}f_tv_t:  \Omega $  $\rightarrow \mathbb{A}  $ with real-valued integrable  components  $f_t$, define
$$\int_{\Omega} f dV:=\sum_{t=0}^{d}v_t \big(\int_{\Omega}f_t dV\big),  $$
where $dV$ stands  for  the   classical   Lebesgue  volume element  in $\mathbb{R}^{(m+1)n}$.
\end{definition}

\begin{proposition}\label{absolute-domain}
Let $f, g$ be as in Definition \ref{integral}. The following properties hold true for all $a \in \mathbb{A}$:\\
(i)  $\int_{\Omega} f dV=\int_{\Omega_1} f dV+\int_{\Omega_2} f dV$, where $ \Omega=\Omega_1 \sqcup \Omega_2$;\\
(ii) $\int_{\Omega} (f+g) dV= \int_{\Omega} f dV + \int_{\Omega} gdV;$\\
(iii) $\int_{\Omega} (af ) dV=a(\int_{\Omega} f dV),  \int_{\Omega} ( fa) dV=(\int_{\Omega} f dV)a;$\\
(iv)  $|\int_{\Omega} f dV| \leq\int_{\Omega} |f| dV$.
\end{proposition}
\begin{proof}
 The results in (i) and (ii) do not involve associativity and  hold naturally.

(iii)  Let $P=\sqcup_{i} P_i$ be a partition of the domain  $\Omega$ and $\lambda(P)$ be the mesh of the partition $P$. By definition, we have, for the  integrable function $f=\sum_{t=0}^{d}f_tv_t: \Omega\rightarrow \mathbb{A}$,
$$\int_{\Omega} f dV=\sum_{t=0}^{d}v_t\Big(\lim_{\lambda(P)\rightarrow 0} \sum_i f_t(x_i) |P_i|\Big), $$
where  $x_i\in P_i$ and $|P_i|$ is the Lebesgue  volume of $P_i$.\\
Note that $d$ is finite, we get
$$\int_{\Omega} f dV=\lim_{\lambda(P)\rightarrow 0} \sum_i \Big( \sum_{t=0}^{d}v_t f_t(x_i)\Big) |P_i|=\lim_{\lambda(P)\rightarrow 0} \sum_i f(x_i) |P_i|.  $$
Hence, observing that $|P_i|$ are real,   we have
$$\int_{\Omega} (af) dV=\lim_{\lambda(P)\rightarrow 0} \sum_i (af(x_i)) |P_i|=a\Big(\lim_{\lambda(P)\rightarrow 0} \sum_i  f(x_i) |P_i|\Big)=a\Big(\int_{\Omega} f dV\Big),$$
and
$$\int_{\Omega} ( fa) dV=\lim_{\lambda(P)\rightarrow 0} \sum_i ( f(x_i)a) |P_i|= \Big(\lim_{\lambda(P)\rightarrow 0} \sum_i  f(x_i) |P_i|\Big) a= \Big(\int_{\Omega} f dV \Big) a.$$

(iv)  We only need to consider the case $\int_{\Omega} |f| dV<+\infty$.  From the proof of (iii), it holds that
$$\int_{\Omega} f dV =\lim_{\lambda(P)\rightarrow 0} \sum_i f(x_i) |P_i|.  $$
Hence,
$$ |\int_{\Omega} f dV| =\lim_{\lambda(P)\rightarrow 0} |\sum_i f(x_i) |P_i|| \leq
\lim_{\lambda(P)\rightarrow 0}\sum_i |f(x_i)| |P_i| =\int_{\Omega} |f| dV,  $$
which completes the proof.
\end{proof}
Similarly, we can define integrals over   surfaces.
 \begin{definition}\label{integral-surface}
Let $\Gamma$ be a  surface in $M^{n}$. For $f=\sum_{t=0}^{d}f_tv_t : \Gamma\rightarrow \mathbb{A}  $ with real-valued integrable  components  $f_t$, define
$$\int_{\Gamma} f dS:=\sum_{t=0}^{d}v_t \big(\int_{\Gamma}f_t dS\big),  $$
where $dS$ stands  for  the   classical   Lebesgue  surface element  in $\mathbb{R}^{(m+1)n}$.
\end{definition}
 We also have following  basic properties as in Proposition \ref{absolute-domain}.
 \begin{proposition}\label{absolute-boundary}
Let $f, g$ be as in Definition \ref{integral-surface}. The following properties hold true for all $a\in \mathbb{A}$:\\
(i)  $\int_{\Gamma} f dS=\int_{\Gamma_1} f dS+\int_{\Gamma_2} f dS$, where $\Gamma=\Gamma_1 \sqcup \Gamma_2$;\\
(ii) $\int_{\Gamma} (f +g) dS= \int_{\Gamma} f dS +\int_{\Gamma} gdS;$\\
(iii) $\int_{\Gamma} (af) dS=a(\int_{\Gamma} f dS), \int_{\Gamma} (fa) dS= (\int_{\Gamma} f dS)a$;\\
(iv)  $  |\int_{\Gamma} f dS | \leq \int_{\Gamma} |f| dS.$
\end{proposition}

\section{Bochner-Martinelli formula}

For $j\in \{1,\ldots,n\}$,  define the functions $K_{j} : M^{n}\setminus\{0\} \rightarrow M$ as
\begin{eqnarray*}\label{BMK-Component}
K_{j}(x) :=\frac{1}{\sigma_{(m+1)n}}\frac{ x_{j} ^{c} }{|x|^{(m+1)n}},
\end{eqnarray*}
where $\sigma_{(m+1)n} $ is the surface area of the unit $((m+1)n-1)$-sphere in $\mathbb{R}^{(m+1)n}$.

Given $x\in M^{n}$,  define the $((m+1)n-1)$-form in $M^{n}\setminus\{x\}$, called   the  Bochner-Martinelli  kernel, as
\begin{eqnarray}\label{BMK}
K(x,y) \equiv K(y-x):=\sum\limits_{j=1}^n K_j(y-x)d\Sigma_{y_j},
\end{eqnarray}
where
$$d\Sigma_{y_j}=(-1)^{(j-1)(m+1)} d y_1 \wedge d y_2 \wedge\cdots\wedge d y_{j-1} \wedge d y[j]\wedge d y_{j+1} \wedge\cdots\wedge  d y_n,$$
with
 \begin{eqnarray*}
 d y_j &=&dy_{j,0}\wedge dy_{j,1}\wedge\cdots \wedge dy_{j,m},
 \\
 d y[j] &=&\sum\limits_{s=0}^m (-1)^{s}v_sd\widehat{y_{j,s}},
 \\
d\widehat{y_{j,s}}&=&dy_{j,0}\wedge\cdots\wedge dy_{j,s-1 }\wedge dy_{j, s+1 }\wedge\cdots\wedge dy_{j,m}.
\end{eqnarray*}

A remarkable feature of the  Bochner-Martinelli  kernel is    its universality  in the sense that it  does not depend on the choice of domains.
 It   should be pointed out that the  Bochner-Martinelli  kernel is  $C^{\infty}$ for $y\neq x$  and has a singularity of order $(m+1)n-1$ at $y=x$.
When $n= 1$,    the Bochner-Martinelli kernel is nothing but the Cauchy kernel.

\begin{example}\label{Cauchy-kernel-example}
 The Cauchy kernel is given by
$$E(x):=\frac{1}{\sigma_{m+1}}\frac{x^{c} }{|x|^{m+1}}, \quad x \in\Omega=M\setminus \{0\},$$
where $\sigma_{m+1}=2\frac{\Gamma^{m+1}(\frac{1}{2}) }{\Gamma (\frac{m+1}{2})} $ is the surface area of the unit ball in $\mathbb{R}^{m+1}$. \\
Then $E \in \mathcal {M}^{L}(\Omega,  M) \cap{\mathcal{M}}^{R}(\Omega,  M).$
\end{example}
 The fact that the Cauchy kernel is left (and right) monogenic  is omitted here since its proof  could  be included in the   proof of Lemma \ref{BM-sum-0} below.

We consider   a bounded  domain   $\Omega$   in $M^{n}$ with piecewise smooth boundary $\Gamma:=\partial \Omega$ and denote by
 $$\nu(y)=(\nu_1(y), \ldots, \nu_n(y)): \Gamma \rightarrow  M^{n}$$
  the unit exterior normal to $\Gamma $ at $y$.
As is known, the  Bochner-Martinelli  kernel in  (\ref{BMK}) takes the form of
$$K(x,y)  = \sum_{j=1}^{n} K_{j}(y-x) \nu_j(y)dS(y),  $$
where $dS$   stands  for  the   classical  scalar element  of surface area on $\Gamma$ in $\mathbb{R}^{(m+1)n}$.

 For a function $f:\Omega \rightarrow \mathbb{A}$, define
 $$K(x,y)\ast f(y):=\sum\limits_{j=1}^n K_j(y-x)(d\Sigma_{y_j}f(y)),$$
equivalently,
 $$K(x,y)\ast f(y) = \sum_{j=1}^{n} K_{j}(y-x) (\nu_j(y)f(y))dS(y). $$

Now we can state  the Bochner-Martinelli theorem for  smooth functions of several hypercomplex variables over real alternative $\ast$-algebras as follows,  which  subsumes the  Clifford case  \cite[Theorem 3.5]{Ren-Wang-14} and the octonionic case \cite[Theorem 7.8]{Wang-Ren-14}.
 \begin{theorem}[Bochner-Martinelli]\label{Bochner-Martinelli-theorem}
Let    $\Omega$ be a bounded domain in $M^{n}$ with piecewise smooth boundary $\Gamma$. If $f\in C^1( \overline{\Omega}, \mathbb{A}) $, then
 \begin{eqnarray*}
 \int_{\Gamma}  K(x,y)\ast f(y) -
\sum_{j=1}^{n}\int_{\Omega} K_{j}(y-x)  (\overline{\partial}_{j} f(y)) dV(y)=
\left\{
\begin{array}{ll}
f(x), \quad  x \in \Omega,
\\
0, \quad \quad x \in \Omega^{-},
\end{array}
\right.
\end{eqnarray*}
where    $dV(y)=d y_1 \wedge d y_2 \wedge\cdots\wedge  d y_n$ stands  for  the   classical   Lebesgue   volume element  in $\mathbb{R}^{(m+1)n}$ and $\Omega^{-}:=M^{n}\setminus \overline{\Omega}$.
\end{theorem}

When $n=1,$  Theorem \ref{Bochner-Martinelli-theorem} reduces  to  the Cauchy-Pompeiu (or called Cauchy-Green) integral  formula \cite[Theorem 3.6]{HRX}.
\begin{theorem}[Cauchy-Pompeiu]\label{Cauchy-Pompeiu}
Let    $\Omega$ be a bounded domain in $M$ with piecewise smooth boundary $\Gamma$. If $f\in C^1( \overline{\Omega}, \mathbb{A}) $, then
 \begin{eqnarray*}\label{CP-theorem}
\int_{\Gamma}  E(y-x) (\nu(y)f(y)) dS(y)-
\int_{\Omega} E (y-x)  (\overline{\partial} f(y)) dV(y)=
\left\{
\begin{array}{ll}
f(x), \quad  x \in \Omega,
\\
0, \quad \quad x \in \Omega^{-},
\end{array}
\right.
\end{eqnarray*}
where  $\nu(y)$ is the unit exterior normal to $\Gamma$ at $y$,
 $dS$ and  $dV$ stand  for  the   classical   Lebesgue surface  element and volume element  in $\mathbb{R}^{m+1}$, respectively.
\end{theorem}

As a special case of Theorem \ref{Bochner-Martinelli-theorem}, we obtain an integral   representation formula for monogenic functions.  In  the quaternionic setting, the first corresponding result   was obtained by Pertici \cite{Pertici-88}  (see also \cite[Theorem 3.3.3]{Colombo-Sabadini-Sommen-Struppa-04}).
 \begin{corollary} \label{Bochner-Martinelli-coro}
Let    $\Omega$ be a  bounded  domain in $M^{n}$ with piecewise smooth boundary $\Gamma$. If $f\in \mathcal{M}( \overline{\Omega}, \mathbb{A})$, then
\begin{eqnarray*}
 \int_{\Gamma}  K(x,y)\ast f(y)  =
\left\{
\begin{array}{ll}
f(x), \quad  x \in \Omega,
\\
0, \quad \quad x \in \Omega^{-}.
\end{array}
\right.
\end{eqnarray*}
\end{corollary}

To prove Theorem \ref{Bochner-Martinelli-theorem}, we need some  technical lemmas. The first lemma is from  \cite[Lemma 4.3]{Xu-Sabadini-25-O}, where a more general version was proved; see  also \cite[Lemma 3.3]{HRX} and \cite[Theorem 1]{Li-Peng-02}. For completeness,  following  the proof as in \cite[Lemma 4.3]{Xu-Sabadini-25-O}  we give the details here.
 \begin{lemma}\label{Cauchy-formula-lemma}
Let   $\Omega $ be a   domain in $M^{n}$ and $j \in \{1,\ldots,n\}$. If $\phi =\sum_{s=0}^m \phi_s  v_s \in C^1(\Omega, M) $  satisfies
\begin{equation}\label{i-j}
\partial_{j,s}\phi_t= \partial_{j,t}\phi_s, \quad 1\leq s\leq m,   1\leq  t\leq m-1,
\end{equation}
 then  for any $a\in \mathbb{A}$
$$\sum_{s=1}^{m}[v_s, \overline{\partial}_{j} \phi_s, a]=0.$$
\end{lemma}
\begin{proof}
By definition, we have for  all $a \in \mathbb{A}$
$$\sum_{s=1}^{m}[v_s, \overline{\partial}_{j} \phi_s, a]
= \sum_{s =1}^{m} \sum_{ t=1}^{m-1} [v_s, v_t, a] \partial_{j,t} \phi_{s} +\sum_{s=1}^{m-1}[v_s, v_m, a]\partial_{j,m} \phi_{s}.
 $$
In view of (\ref{i-j}), it holds that
$$\sum_{1\leq s,t\leq m-1}   [v_s,v_t, a]\partial_{j,t} \phi_s =0.$$
Hence,
\begin{eqnarray*}
\sum_{s=1}^{m}[v_s, \overline{\partial}_{j} \phi_s, a]
&=&  \sum_{ t=1}^{m-1} [v_m, v_t, a] \partial_{j,t} \phi_{m} +\sum_{s=1}^{m-1}[v_s, v_m, a]\partial_{j,m} \phi_{s}
\\
&=&    \sum_{ s=1}^{m-1} [v_m, v_s, a] \partial_{j,s} \phi_{m} +\sum_{s=1}^{m-1}[v_s, v_m, a]\partial_{j,m} \phi_{s}
\\
 &=&\sum_{ s=1}^{m-1} [v_m, v_s, a]( \partial_{j,s} \phi_{m}-\partial_{j,m} \phi_{s}),\end{eqnarray*}
which vanishes from  (\ref{i-j}).  The proof is complete.
\end{proof}


\begin{lemma}\label{BM-sum-0}
Given $x\in M^{n}$, we have for all $y\in M^{n}\setminus\{x\}$
$$\sum_{j=1}^{n} \overline{\partial}_{j}K_{j}(y-x) =\sum_{j=1}^{n}K_{j}(y-x)\overline{\partial}_{j}=0.$$
\end{lemma}
\begin{proof}
Given $x\in M^{n}$, direct calculations give that for all $y\in M^{n}\setminus\{x\}$ and $j\in \{1,\ldots,n\}$
\begin{eqnarray*}
&&\sigma_{(m+1)n} \overline{\partial}_{j}K_{j}(y-x)  \\
&=& \sum_{s,t=0}^{m} v_s v_t^{c} \Big( \frac{\delta_{s,t}}{|y-x|^{(m+1)n}} -(m+1)n \frac{(y_{j,t}-x_{j,t})(y_{j,s}-x_{j,s}) }{|y-x|^{(m+1)n+2} } \Big)\\
 &=&\sum_{s}^{m} v_s v_s^{c}  \frac{1}{|y-x|^{(m+1)n}} - (m+1)n \sum_{s,t=0}^{m}  \frac{ (y_{j,s}-x_{j,s})v_s (y_{j,t}-x_{j,t})v_t^{c}}{|y-x|^{(m+1)n+2} }
 \\
 &=&   \frac{m+1}{|y-x|^{(m+1)n}} - (m+1)n \sum_{s,t=0}^{m}  \frac{  (y_{j}-x_{j}) (y_{j}-x_{j})^{c}}{|y-x|^{(m+1)n+2} }.
 \end{eqnarray*}
Hence,
 $$\overline{\partial}_{j}K_{j}(y-x) =\frac{ m+1 }{\sigma_{(m+1)n}}
\Big( \frac{1}{|y-x|^{(m+1)n}}- \frac{n|y_{j}-x_{j}|^{2} }{|y-x|^{(m+1)n+2}} \Big). $$
Taking   summation over $j$ would  conclude  that $\sum_{j=1}^{n} \overline{\partial}_{j}K_{j}(y-x) =0$.  The fact $\sum_{j=1}^{n} K_{j}(y-x) \overline{\partial}_{j}=0$ can be obtained by  similar calculations.  The proof is complete.
\end{proof}
\begin{remark}\label{kernel-harmonic-remark}
From the proof of Lemma \ref{BM-sum-0},  the Bochner-Martinelli  kernel given by (\ref{BMK}) is not left (or right) monogenic unless $n=1$, but still harmonic with respect to $x$, which  follows from the fact
$$K_{j}(x)= \partial_{j} G(x),\quad j=1,\ldots,n,$$
where $\partial_{j} =\sum _{s =0}^{m}  v_{s}^{c}   \partial _{x_{j,s}}$ is the conjugated  Dirac operator corresponding to the $j$-th copy of $M$ and $G$ is the fundamental solution for the Laplacian in $\mathbb{R}^{(m+1)n}$ given by
\begin{eqnarray*}
G(x) =
\left\{
\begin{array}{ll}
- \frac{1}{2\pi} \log |x |, \quad \quad \quad \quad \quad \quad \quad \quad \quad  m=n=1,
\\
\frac{1}{(-(m+1)n+2)\sigma_{(m+1)n}}   \frac{1}{|x|^{(m+1)n-2}}, \quad (m+1)n>2.
\end{array}
\right.
\end{eqnarray*}
\end{remark}

\begin{lemma}\label{Stokes-A}
Let     $\Omega$ be a bounded domain in $M^{n}$ with piecewise smooth boundary $\Gamma$. If  $ \phi  \in C^1( \overline{\Omega},  M)$ and $f\in C^1( \overline{\Omega},\mathbb{A})$, then for $j= 1,\ldots,n,$
 $$\int_{\Gamma}   \phi    (\nu_{j} f)dS= \int_{\Omega} \Big((\phi \overline{\partial}_{j}  )f + \phi (\overline{\partial}_{j} f)
  - \sum_{s=0}^{m} [v_s, \overline{\partial}_{j} \phi_s, f]\Big)dV,
$$
 where    $dV$ stands  for  the   classical   Lebesgue   volume element  in $\mathbb{R}^{(m+1)n}$.
\end{lemma}
\begin{proof}
Let $\phi=\sum_{s=0}^{m}\phi_sv_s \in C^1( \overline{\Omega},  M)$, $f=\sum_{t=0}^{d}f_tv_t \in C^1( \overline{\Omega},  \mathbb{A})$  and $\nu=(\nu_1, \ldots, \nu_n)$ be the unit exterior normal to $\partial \Omega$ with $\nu_{j}=\sum_{k=0}^{m} v_k\nu_{j,k}$ for $j=1,\ldots,n$.
By the Stokes formula, it holds that for all real-valued functions $\phi_{s}, f_t\in C^1( \overline{\Omega})$
$$\int_{\Gamma}   \phi_{s}f_t  \nu_{j,k} dS
=\int_{\Omega}(   (\partial_{j,k}\phi_{s} )f_t + \phi_{s} (\partial_{j,k}f_t) ) dV,\quad k=0,1,\ldots,m. $$
By multiplying by $v_k, k=0,1,\ldots,m,$  on both sides of the  formula above, respectively, and  taking summation over $k$, we get
$$\int_{\Gamma}   \phi_{s}f_t   \nu_{j} dS =
\int_{ \Omega} ( (\overline{\partial}_{j}  \phi_{s} )f_t+ \phi_{s} (\overline{\partial}_{j}  f_t) )dV.$$
Then,  by multiplying by $v_t, t=0,1,\ldots,d,$ from the right of both sides above  and taking   summation over $t$,
$$\int_{\Gamma}   \phi_s    \nu_{j} fdS =
\int_{\Omega}  ( (  \overline{\partial}_{j}\phi_s )f + \phi_s  (\overline{\partial}_{j}  f) )dV,$$
By multiplying by  $v_s, s=0,1,\ldots,m,$  from the   left of both sides above,  and then taking summation over $s$, there holds
\begin{eqnarray*}
  \int_{\Gamma}   \phi (    \nu_{j} f)dS
&=& \int_{\Omega}  (\sum_{s=0}^{m}  v_s ((  \overline{\partial}_{j}\phi_s )f) + \phi  (\overline{\partial}_{j}  f) )dV
\\
&=&\int_{ \Omega}\Big( (\phi \overline{\partial}_{j}  )f  + \phi (\overline{\partial}_{j} f) -
\sum_{s=0}^{m} [ v_s,  \overline{\partial}_{j} \phi_s ,f]\Big)dV
\\
&=&\int_{ \Omega}\Big( (\phi \overline{\partial}_{j}  )f  + \phi (\overline{\partial}_{j} f) -
\sum_{s=1}^{m} [ v_s,  \overline{\partial}_{j} \phi_s ,f]\Big)dV,
 \end{eqnarray*}
which completes the proof.
 \end{proof}

Denote by $B(x,\epsilon):=\{y \in M^{n}: | y- x|<\epsilon\}$  the ball centered at $x$ with radius $\epsilon>0$  and by $S(x,\epsilon):=\partial B(x,\epsilon)$  its boundary. Now we come to the proof of the Bochner-Martinelli theorem in several hypercomplex variables over real alternative $\ast$-algebras as follows.

\begin{proof}[Proof of Theorem \ref{Bochner-Martinelli-theorem}]
Fix $ x \in M^{n}$ and consider  $B(x,\epsilon)$. Set $\phi_{j}( y)=K_{j}( y-x)=\sum_{s=0}^m  \phi_{j,s}(y)v_s$ for $j=1,\ldots,n$. Observe that we have  by Lemma \ref{BM-sum-0}
\begin{equation}\label{sum-zero}
 \sum_{j=1}^{n} \phi_j( y) \overline{\partial}_{j} =0, \quad    y\in M^{n} \setminus \{x\}, \end{equation}
and each $\phi_{j} $ satisfies conditions in (\ref{i-j}), which gives that by Lemma \ref{Cauchy-formula-lemma}, for all $a\in \mathbb{A}$,
$$\sum_{s=1}^{m}[v_s, \overline{\partial}_{j}  \phi_{j,s}(y), a]=0, \quad    y\in M^{n} \setminus \{x\}.$$

When $x \in \Omega$,   for $\epsilon$ small enough, we have by  Lemmas \ref{Stokes-A} for $f \in C^1( \overline{\Omega},  \mathbb{A})$
 \begin{eqnarray}\label{Stokes-ball}
&& \int_{\Gamma} \phi_{j}    ( \nu_{j} f)dS - \int_{S( x,\epsilon)}   \phi_{j}    (\nu_{j} f)dS
\notag
\\
&=& \int_{ \Omega \setminus B( x,\epsilon)}\Big( (\phi_j \overline{\partial}_{j}  )f + \phi_j (\overline{\partial}_{j}f) -
\sum_{s=1}^{m} [v_s,  \overline{\partial}_{j} \phi_{j,s}, f]\Big)dV \notag
\\
&=&\int_{ \Omega \setminus B( x,\epsilon)}\big( (\phi_j \overline{\partial}_{j}  )f + \phi_j (\overline{\partial}_{j}f)\big)dV.
 \end{eqnarray}
Notice that the unit exterior normal to $S( x,\epsilon)$ at $y$ is given by $\nu(y)=(\nu_1(y), \ldots, \nu_n(y))$ with  $\nu_j(y)= \frac{ y_{j}- x_{j}}{| y- x|}$ for $j=1,\ldots,n$. Hence, recalling Proposition \ref{artin}, it follows that
 \begin{eqnarray*}
 \int_{S( x,\epsilon) }   \phi_j    (\nu_{j} f)dS
&=&  \frac{1}{\sigma_{(m+1)n}}\int_{S( x,\epsilon) }  \frac{  (y_{j}- x_{j})^{c}}{| y- x|^{(m+1)n}}   \Big(\frac{ y_{j}- x_{j}}{| y- x|} f( y)\Big)dS( y)
\\
&=& \frac{1}{\sigma_{(m+1)n}}\int_{S( x,\epsilon) } \Big( \frac{(y_j- x_j)^{c}}{| y- x|^{(m+1)n}}   \frac{ y_j- x_j}{| y- x|}\Big) f( y)dS( y)
\\
&=& \frac{1}{\epsilon^{(m+1)n+1}\sigma_{(m+1)n}}\int_{S( x,\epsilon) } |y_j- x_j|^{2}  f( y)dS( y).
\end{eqnarray*}
Inserting  this result into (\ref{Stokes-ball}), we have  for $j=1,\ldots,n,$
 \begin{eqnarray*}
&& \int_{\Gamma} \phi_{j}    ( \nu_{j} f)dS - \frac{1}{\epsilon^{(m+1)n+1}\sigma_{(m+1)n}}\int_{S( x,\epsilon) } |y_j- x_j|^{2}  f( y)dS( y)\\
&=& \int_{ \Omega \setminus B( x,\epsilon)}\big( (\phi_j \overline{\partial}_{j}  )f + \phi_j (\overline{\partial}_{j}f) \big)dV.
\end{eqnarray*}
Then, taking summation over $j$ and by (\ref{sum-zero}), we have
 \begin{eqnarray*}
&& \int_{\Gamma} \sum_{j=1}^{n}  \phi_{j}    ( \nu_{j} f)dS - \frac{1}{\epsilon^{(m+1)n+1}\sigma_{(m+1)n}}\int_{S( x,\epsilon) } \sum_{j=1}^{n} |y_j- x_j|^{2}  f( y)dS( y)\\
&=& \int_{ \Omega \setminus B( x,\epsilon)}\sum_{j=1}^{n}  \big( (\phi_j \overline{\partial}_{j}  )f + \phi_j (\overline{\partial}_{j}f) \big)dV
\\
&=& \int_{ \Omega \setminus B( x,\epsilon)}\sum_{j=1}^{n}   \phi_j (\overline{\partial}_{j}f)  dV,
\end{eqnarray*}
i.e.
\begin{equation}\label{aa}
\frac{1}{\epsilon^{(m+1)n-1}\sigma_{(m+1)n}}\int_{S( x,\epsilon) } f( y)dS( y)
= \int_{\Gamma } \sum_{j=1}^{n}  \phi_{j}    ( \nu_{j} f)dS -
 \int_{ \Omega \setminus B( x,\epsilon)}\sum_{j=1}^{n}   \phi_j (\overline{\partial}_{j}f)  dV.\end{equation}
Notice that for $f\in C^1(\overline{\Omega}, \mathbb{A})$ and  fixed $x\in \Omega$
$$\lim_{\epsilon\rightarrow 0}  \int_{   B( x,\epsilon)}\sum_{j=1}^{n}   \phi_j (\overline{\partial}_{j}f)  dV
= 0,$$
and
$$\lim_{\epsilon\rightarrow 0}  \frac{1}{\epsilon^{(m+1)n-1}\sigma_{(m+1)n}}\int_{S( x,\epsilon) } f( y)dS( y) = f( x).$$
Finally, letting  $\epsilon\rightarrow0$ in (\ref{aa}), we conclude that
 $$ f(x)=\int_{\Gamma}  K(x,y)\ast f(y) -
\sum_{j=1}^{n}\int_{\Omega} K_{j}(y-x)  (\overline{\partial}_{j} f(y)) dV(y), \quad  x \in \Omega. $$

When $ x \in \Omega^{-}$, formula  (\ref{Stokes-ball}) should be replaced by
 \begin{eqnarray*}\label{Stokes-out}
\int_{\Gamma} \phi_{j}    ( \nu_{j} f)dS
&=& \int_{ \Omega }\Big( (\phi_j \overline{\partial}_{j}  )f + \phi_j (\overline{\partial}_{j}f) -
\sum_{s=1}^{m} [v_s,  \overline{\partial}_{j} \phi_{j,s}, f]\Big)dV \notag
\\
&=&\int_{ \Omega }\big( (\phi_j \overline{\partial}_{j}  )f + \phi_j (\overline{\partial}_{j}f)\big)dV.
 \end{eqnarray*}
Taking summation over $j$ and recalling (\ref{sum-zero}), we  get
 $$ \int_{\Gamma}  K(x,y)\ast f(y) -
\sum_{j=1}^{n}\int_{\Omega} K_{j}(y-x)  (\overline{\partial}_{j} f(y)) dV(y)=0, \quad  x \in \Omega. $$
The proof is complete.
\end{proof}

\section{Plemelj-Sokhotski formula}
In this section, we shall denote by $C$  some positive constant, which may vary in different places and is independent  of the variables such as $x,y,x'$.

Let $E$ be a set in $M^{n}$ and  $\alpha\in (0,1)$. Denote by $\Lambda^{\alpha}(E,\mathbb{A})$ all functions   $f:E\rightarrow \mathbb{A}$ which satisfy the  H\"{o}lder condition with exponent $\alpha$, i.e.
$$  |f(x)-f(x')|\leq C|x-x'|^{\alpha},\quad x,x'\in E,$$
where $C$ denotes some constant.

Let     $\Omega$ be a bounded domain in $M^{n}$ with boundary $\Gamma$.
For $x\in \Gamma$, we define the solid angle of the tangent cone to the surface $\Gamma$ at $x$ by
$$\tau(x)=\lim_{\epsilon\rightarrow 0} \frac{\mathrm{vol} \{S(x,\epsilon)\cap \Omega\} }{\mathrm{vol} \{S(x,\epsilon) \}},$$
provided it exists. When $\Gamma$ is piecewise smooth, one can  follow literally the proof of \cite[Lemma 1]{Lin-88} and show that $\tau$ is well-defined   and takes values in $(0,1)$. In particular,  $\tau(x)=1/2$ for smooth $\Gamma$.

For $f\in \Lambda^{\alpha}(\Gamma, \mathbb{A})$, we define the Bochner-Martinelli  type integral
\begin{eqnarray}\label{Bochner-matinelli-definition}
  C_{\Gamma}[f](x)=\int_{\Gamma }K(x,y) \ast f(y), \quad x\in  M^n.
  \end{eqnarray}
Since the Laplace operator is   real-scalar valued,  it is easy to prove from Remark \ref{kernel-harmonic-remark} that $C_{\Gamma}[f]$ is   harmonic both in $\Omega$  and in $\Omega^{-}$. Moreover, $C_{\Gamma}[f](x)=O(|x|^{1-n(m+1)})$ as $|x|\rightarrow \infty$ by Propositions \ref{absolute-boundary} and \ref{norm}.  When $x\in \Gamma$, the integrand in (\ref{Bochner-matinelli-definition}) has the singularity $|y-x|^{1-n(m+1)}$  and the integral  does not exist generally as an improper integral. Hence, when $x\in \Gamma$,    the singular Bochner-Martinelli integral is defined in the sense of Cauchy's principle value as
$$ C_{\Gamma}[f](x)=\mathrm{P. V.} \int_{\Gamma }K(x,y) \ast f(y)
      =\lim_{\epsilon\rightarrow 0} \int_{\Gamma \setminus B( x,\epsilon)} K(x,y) \ast f(y), \quad x\in  \Gamma.$$

\begin{lemma}\label{Bochner-Martinel-constant}
For any $a\in \mathbb{A}$,
$$C_{\Gamma}[a](x)=\tau(x)a, \quad x \in \Gamma. $$
\end{lemma}
\begin{proof}
By definition, it holds that
$$C_{\Gamma}[a](x)=\mathrm{P. V.} \int_{\Gamma }K(x,y)\ast a
      =\lim_{\epsilon\rightarrow 0} \int_{\Gamma \setminus B( x,\epsilon)} K(x,y)\ast a, \quad x\in  \Gamma.$$
Note that
$$   \Gamma \setminus B( x,\epsilon)=\partial ( \Omega \setminus B( x,\epsilon)) \cup S^{+}( x,\epsilon),  $$
where $S^{+}( x,\epsilon)= \Omega  \cap  S( x,\epsilon) $ is the part of the sphere  $S( x,\epsilon)$ lying in  $\Omega$.
Due to that any  constant $a\in \mathbb{A}$ is  in $\mathcal {M}(\Omega \setminus B( x,\epsilon), \mathbb{A})  \cap C(\overline{\Omega \setminus B( x,\epsilon)}, \mathbb{A})$,  we have by Corollary  \ref{Bochner-Martinelli-coro}
$$\int_{\partial (\Omega \setminus B( x,\epsilon))} K(x,y)\ast a =0.$$
Hence,
 \begin{eqnarray*}
  \int_{\Gamma \setminus B( x,\epsilon)} K(x,y)\ast a
&=& \int_{\partial ( \Omega \setminus B( x,\epsilon)) } K(x,y)\ast a+\int_{S^{+}( x,\epsilon)} K(x,y)\ast a
\\
&=& \int_{S^{+}( x,\epsilon)} K(x,y)\ast a \\
&=& \frac{1}{\sigma_{(m+1)n}} \int_{S^{+}( x,\epsilon)} \sum_{j=1}^{n}   \frac{  (y_{j}- x_{j})^{c}}{| y- x|^{(m+1)n}}  \Big(  \frac{ y_{j}- x_{j}}{| y- x|} a\Big) dS( y)\\
&=& \frac{1}{\sigma_{(m+1)n}} \int_{S^{+}( x,\epsilon)} \sum_{j=1}^{n}  \Big( \frac{ (y_{j}- x_{j})^{c}}{| y- x|^{(m+1)n}}  \frac{ y_{j}- x_{j}}{| y- x|} \Big) a dS( y)\\
&=& \frac{\mathrm{vol}\{S^{+}( x,\epsilon)\}}{\sigma_{(m+1)n}\epsilon^{(m+1)n-1}} a\\
&=& \frac{\mathrm{vol}\{S^{+}( x,\epsilon)\}}{\mathrm{vol}\{ S( x,\epsilon)\}} a,
\end{eqnarray*}
where the fourth equality follows from Proposition \ref{artin}.
The proof is complete.
\end{proof}

To prove our next result,  we recall a  technical result from \cite[Lemma 1]{Lin-88}.
\begin{lemma} \label{Lin-lemma}
 Let     $\Omega$ be a bounded domain in $M^{n}$ with piecewise smooth boundary $\Gamma$ and $x\in \Gamma$ be fixed. If $x'$ tends to $x$ along arbitrarily non-tangential direction from $\Omega$ (or $\Omega^{-}$),  there exist $\delta>0$ and $N\geq2$, which are independent of $x$, such that, for  $|x'-x|\leq \delta$,
 $$|x'-x|\leq N|x'-y|,\ |y-x|\leq (N+1)|x'-y|, \quad \forall \ y\in \Gamma.$$

 \end{lemma}

  As we know,  monogenic functions of  one hypercomplex variable over real alternative $\ast$-algebras  would  include  the $\psi$-hyperholomorphic functions  associated with the \textit{structural set} $\psi$ \cite[Example 2.7]{HRX}. For some   Plemelj-Sokhotski formulas  on $\psi$-hyperholomorphic functions, see e.g. \cite{Uwe} and references therein. Now we   can establish the Plemelj-Sokhotski formula for monogenic functions of  several  hypercomplex variables over real alternative $\ast$-algebras as follows.

\begin{theorem}[Plemelj-Sokhotski]\label{Plemelj-Sokhotski-theorem}
  Let  $\Omega$ be a bounded domain in $M^{n}$ with piecewise smooth boundary $\Gamma=\partial\Omega$ and let $f\in \Lambda^{\alpha}(\Gamma, \mathbb{A})$, where $0<\alpha <1$. Then the Bochner-Martinelli integral $C_{\Gamma}$ extends continuously to $\Gamma$ along arbitrarily non-tangential (n.t. for short) direction from   $\Omega$ (or  $\Omega^{-}$). Moreover,   for $x\in \Gamma$,
the following  two limits   exist
  $$C_{\Gamma}^{+}[f](x):= \mathrm{n.t.}\lim_{ \Omega \ni x' \rightarrow x} C_{\Gamma} [f](x'),$$
   $$C_{\Gamma}^{-}[f](x):=\mathrm{n.t.}\lim_{  \Omega^{-} \ni x' \rightarrow x} C_{\Gamma} [f](x'),$$
and  are given by, respectively,
$$ C_{\Gamma}^{+}[f](x)=C_{\Gamma}[f](x) + (1-\tau(x))f(x), $$
$$  C_{\Gamma}^{-}[f](x)=C_{\Gamma}[f](x) -\tau(x)f(x).$$
 \end{theorem}

 \begin{proof}
Let $x$ be a fixed point on $\Gamma$. By definition,
we have
$$ C_{\Gamma}[f](x')=\int_{\Gamma} K(x',y) \ast f(y)
 =\Psi(x')+\int_{\Gamma}K(x',y) \ast f(x), \quad x'\in M^{n},$$
where $$\Psi(x')=\int_{\Gamma}K(x',y) \ast (f(y)-f(x)).$$
By Lemma \ref{Bochner-Martinel-constant}, it holds that
\begin{eqnarray*}
 \Psi(x) = \int_{\Gamma}K(x,y) \ast f(y)-\int_{\Gamma}K(x,y) \ast f(x)
= C_{\Gamma}[f](x)-\tau(x)f(x).
\end{eqnarray*}
By Corollary \ref{Bochner-Martinelli-coro}, it holds that
\begin{eqnarray*}
 \int_{\Gamma}K(x',y) \ast  f(x) =
\left\{
\begin{array}{ll} f(x), \quad    x' \in \Omega,
\\
0, \quad \quad   x' \in \Omega^{-}.
\end{array}
\right.
\end{eqnarray*}
From these two facts above,  the assertion would be proved if we can show
$$\mathrm{n.t.}\lim_{ \Omega \ni x' \rightarrow x} \Psi(x') =\Psi(x), \ \mathrm{n.t.}\lim_{ \Omega^{-} \ni x' \rightarrow x} \Psi(x') =\Psi(x).$$

To this end,  let   $x'  \in B(x,\delta)$, where  $\delta>0$ is small enough. Consider the ball $B(x,2\delta)$ and set $\Gamma_{\delta}=\Gamma\cap B(x, 2\delta)$. Now  the integral  $\Psi(x') -\Psi(x)$ can be decomposed into two parts  over $\Gamma_{\delta}$ and  $\Gamma\setminus \Gamma_{\delta}$, i.e.
\begin{eqnarray*}
\Psi(x') -\Psi(x)= \int_{\Gamma}(K(x',y)-K(x,y) ) \ast (f(y)-f(x))=I_1+I_2,
\end{eqnarray*}
where
\begin{eqnarray*}
I_1=\int_{\Gamma_{\delta}}(K(x',y)-K(x,y) ) \ast (f(y)-f(x)),
\\
I_2=
\int_{\Gamma\setminus \Gamma_{\delta}} (K(x',y)-K(x,y) ) \ast (f(y)-f(x)).
\end{eqnarray*}

Firstly, let us   estimate the integral   $I_1$.
By Lemma \ref{Lin-lemma},  there exists $N\geq2$, which is independent of $x$, such that, for  $|x'-x|\leq \delta$ and all $y\in \Gamma$,
\begin{eqnarray}\label{Lin-1}
 |x'-x|\leq N|x'-y|,
 \end{eqnarray}
and  \begin{eqnarray}\label{Lin-2}
  |y-x|\leq (N+1)|x'-y|.
 \end{eqnarray}
Direct computations give
  \begin{eqnarray*}
   &&
   \Big|\frac{(y_j-  x_j')^{c} }{|y-x' |^{(m+1)n}}- \frac{( y_j-  x_j)^{c}}{|y- x'|^{(m+1)n}}\Big| \notag
   \\
 &\leq &\frac{| x_j -x_j'|}{|y-x' |^{(m+1)n}}+|y_j-x_j|
  \Big|\frac{1}{|y-x' |^{n(m+1)}}-\frac{1}{|y-x|^{n(m+1)}}\Big| \notag
  \\
  &\leq&
  \frac{| x_j - x'_j|}{|y-x' |^{(m+1)n}}
  +|y_j -x_j||x -x'|\frac{\sum\limits_{s=0}^{(m+1)n-1}|y-x|^{(m+1)n-1-s}|y-x' |^{s}
  }{|y-x |^{(m+1)n}|y-x'|^{(m+1)n}} \notag
  \\
    &\leq&
 | x - x'| \Big( \frac{1}{|y-x' |^{(m+1)n}}
  +  \frac{\sum\limits_{s=0}^{(m+1)n-1}|y-x|^{(m+1)n-1-s}|y-x' |^{s}
  }{|y-x |^{(m+1)n-1}|y-x'|^{(m+1)n}}\Big) \notag
  \\
      &=&
 | x - x'|  \frac{2|y-x |^{(m+1)n-1} +\sum\limits_{s=1}^{(m+1)n-1}|y-x|^{(m+1)n-1-s}|y-x' |^{s}
  }{|y-x' |^{(m+1)n}|y-x |^{(m+1)n-1}}.
  \end{eqnarray*}
  Hereafter, we  write $A \lesssim   B$  for $A \leq C B$.
From (\ref{Lin-2}), we have
 \begin{eqnarray}\label{BM-two-point}
  2|y-x |^{(m+1)n-1} +\sum\limits_{s=1}^{(m+1)n-1}|y-x|^{(m+1)n-1-s}|y-x' |^{s} \lesssim |y-x'|^{(m+1)n-1}, \end{eqnarray}
which means
 $$| K(x',y)-K(x,y)| \lesssim  \frac{| x - x'|}{|y-x'|} \frac{1}{|y-x|^{(m+1)n-1}} \lesssim  \frac{1}{|y-x|^{(m+1)n-1}},$$
 where the second inequality follows from (\ref{Lin-1}).
Hence, we have  by Propositions \ref{absolute-boundary} and \ref{norm}
 $$|I_1| \lesssim \int_{\Gamma_{\delta}} |y-x|^{\alpha+1-(m+1)n} dS(y) \lesssim \delta^\alpha.$$

 Secondly, we  estimate the integral   $I_2$.
Observe that, for $x'\in B(x,\delta)$ and $y \notin B(x,2\delta)$, there  hold that
$$|x'-x|\leq \frac{1}{2} | y-x|, \ |x'-x|\leq  |y-x'|,$$
 which gives
$$  |y-x|\leq |y-x'|+ |x'-x|\leq 2|y-x'|,$$
i.e.
 $$    |y-x|\leq 2|x'-y|, \quad \forall \ y\in \Gamma\setminus \Gamma_{\delta}.$$
Hence,  (\ref{BM-two-point}) holds also for $y\in \Gamma\setminus \Gamma_{\delta}$, and then
 $$| K(x',y)-K(x,y)| \lesssim  \frac{| x - x'|}{|y-x'|} \frac{1}{|y-x|^{(m+1)n-1}}  \leq   \frac{| x - x'|}{2|y-x|^{(m+1)n}}.$$
Now we can obtain, by Propositions \ref{absolute-boundary} and \ref{norm} again,
$$|I_2|  \lesssim   \delta \int_{\Gamma\setminus \Gamma_{\delta}}  |y-x |^{\alpha -(m+1)n }   dS(y)  \lesssim   \delta \delta^{\alpha-1}= \delta^{\alpha }.  $$
The proof is complete.
\end{proof}

 In particular, by Theorem \ref{Plemelj-Sokhotski-theorem}, we trivially obtain the jump formula.
 \begin{corollary}
  Let  $\Omega$ be a bounded domain in $M^{n}$ with piecewise smooth boundary $\Gamma$ and let $f\in \Lambda^{\alpha}(\Gamma, \mathbb{A})$, where $0<\alpha <1$. Then
$$ C_{\Gamma}^{+}[f](x) -C_{\Gamma}^{-}[f](x)=f(x),\quad  x\in \Gamma.$$
 \end{corollary}

\section{Hartogs theorem}
Over real alternative $\ast$-algebras,  we can deduce a Hartogs  extension  theorem for monogenic functions of several hypercomplex variables.
To this end,  we  first prove  the existence of the compactly supported solution to the  inhomogeneous Cauchy-Riemann equations, which depends on   an  auxiliary result in \cite[Theorem 3.8]{HRX} as follows.
\begin{theorem}\label{Cauchy-Pompeiu-inverse}
Let    $\Omega$ be a bounded domain in $M$ with  smooth boundary  and    $f\in C^1( \overline{\Omega}, \mathbb{A})$. Then the Teodorescu transform  given by
 $$T[f](x)=-\int_{  \Omega}  E(y-x) f(y) dV(y), \quad  x \in \Omega,   $$
is a right inverse   of the Cauchy-Riemann operator $\overline{\partial}$, i.e.
$$\overline{\partial} T[ f] (x) =f(x),\quad x\in \Omega.$$
\end{theorem}
Let $\Omega$ be a domain in $M^n$. Denote by $C_0^k(\Omega,  \mathbb{A})$ the spaces of functions in $C^k(\Omega,  \mathbb{A})$ that have compact supports in $\Omega$.
\begin{theorem}\label{inhomogeneous-theorem}
Let $g_1, g_2,\ldots,g_n\in C_0^k( M^n,  \mathbb{A})$ be of compact support, where $n>1$ and $ k\geq 3$. Then   the following
three statements are equivalent:\\
(i) The inhomogeneous equation system
\begin{eqnarray}\label{non-homogeneous-equation}
\overline{\partial}_jf=g_j, \quad j=1,\ldots,n,
\end{eqnarray}
admits a compactly supported solution $f\in C_0^{k}(M^n, \mathbb{A})$. Furthermore,   the solution $f$ vanishes in the unbounded component  of
$M^n\setminus (\mathrm{supp}g_1\cup\cdots\cup\mathrm{supp}g_n)$.
\\
(ii) For $j=2,\ldots, n$ and $x=(x_1, x^{0})=(x_1,x_2,\ldots, x_n)\in M^n$,
\begin{eqnarray}\label{integral-condition}
g_j(x)= -\overline{\partial}_j  \Big(\int_{M}  E(y_1)f(y_1+x_1,x^{0})  dy_1 \Big).
\end{eqnarray}
(iii) For $i,j=1,\ldots, n$, there holds
\begin{eqnarray}\label{condition-Laplacde}
\overline{\partial}_i (\partial_j g_j)=\triangle_j g_i, \end{eqnarray}
where   $\Delta_{i}$ is the Laplacian    in  $\mathbb{R}^{m+1}$ with respect to the $i$-th variable.
\end{theorem}
\begin{proof}
In the  proof given below, keep in mind that  all functions  $g_1, g_2,\ldots,g_n\in C_0^k( M^n,  \mathbb{A})$ are  of compact supported in $M^{n}$.

(i)$\Longrightarrow$(ii)
Suppose that the inhomogeneous equation system (\ref{non-homogeneous-equation}) has a solution $f \in C_0^{k}(M^n, \mathbb{A})$.
Then, for $x=(x_1, x^{0})=(x_1,x_2,\ldots, x_n)\in M^n$,
\begin{eqnarray*}
\int_{M}  E(y_1)g_1(y_1+x_1,x^{0})  dy_1
&=& \int_{M}  E(y_1-x_1)g_1(y_1,x^{0})  dy_1
\\
&=&\int_{M}  E(y_1-x_1) (\overline{\partial}_1 f)(y_1,x^{0})f dy_1
\\
&=& -f(x_1, x^{0}),
 \end{eqnarray*}
where the last equality follows from  Theorem \ref{Cauchy-Pompeiu} and $f$ is of compact support in $M^n$.
Hence, it holds that for $j=2,\ldots, n$
$$-\overline{\partial}_j  \Big( \int_{M}  E(y_1)f(y_1+x_1,x^{0})  dy_1\Big)= \overline{\partial}_j f(x)=g_j(x),$$
which gives the conclusion in (ii).

(ii)$\Longrightarrow$(i) For $g_1 \in C_0^k( M^n,  \mathbb{A})$,
set
\begin{eqnarray}\label{function-solusion}
f(x)=- \int_{M}  E(y_1)g_1(y_1+x_1,x^{0})  dy_1=-\int_{M}  E(y_1-x_1)g_1(y_1,x^{0})  dy_1,
\end{eqnarray}
which is well-defined on $M^{n}$ and satisfies that  $f \in C^{k}(M^n, \mathbb{A})$.
Whence, if (ii) holds,  then
 $$\overline{\partial}_j f=g_j, \quad j=2,\ldots, n.$$
As for $j=1$, by Theorem \ref{Cauchy-Pompeiu-inverse}, we get
$$\overline{\partial}_1 f(x)= g_1(x_1,x^{0})=g_1(x). $$
Hence, the function $f$ given by (\ref{function-solusion}) is a solution of the inhomogeneous equation system (\ref{non-homogeneous-equation}),
 which implies that    $\overline{\partial}_j f (j=1,\ldots,n)$ vanish on $\Omega_0=M^n\setminus (\mathrm{supp}g_1\cup\cdots\cup\mathrm{supp}g_n)$. That is to say  $f$ is monogenic  and thus, by Remark \ref{harmonic}, harmonic in $\Omega_0$.
 According to (\ref{function-solusion}),  $f(x)=f(x_1, x^{0})$ vanishes in an   unbounded open set in $M^{n}$ where $|x^{0}|$ is large enough.
Hence, from the identity principle for harmonic functions,  $f$ vanishes  in the (connected) unbounded component  of
$\Omega_0$. Consequently,   $f$  is of compact support.

(i)$\Longrightarrow$(iii) Assume that the system (\ref{non-homogeneous-equation}) has a solution $f$, i.e. $\overline{\partial}_jf=g_j$ for $ j=1,2,\ldots,n$.  From \cite[Proposition 5]{Perotti-22}, we have
 $$ \partial_j (\overline{\partial}_jf)  =  (\partial_j \overline{\partial}_j)f  = \triangle_j f,  $$
which gives that, for $i=1,\ldots, n$,
$$  \overline{\partial}_i (\partial_j g_j) =\overline{\partial}_i (\triangle_j f)= \triangle_j (\overline{\partial}_i f) =\triangle_j g_i,$$
where the second equality is due to that $ \triangle_j$ are   real-scalar-valued operators.

(iii)$\Longrightarrow$(i)
As  shown before,    the function $f\in C_0^{k}(M^n, \mathbb{A})$ given by (\ref{function-solusion}) satisfies
$\overline{\partial}_1 f(x)=g_1(x). $ Now  it remains to show $\overline{\partial}_j f=g_j $ for $j=2,\ldots, n,$ when (iii) holds.
In fact,  noticing that $ \triangle_i$ are real-scalar-valued operators for $i=1, \ldots,n$,  we have
\begin{eqnarray*}
  \triangle_i f(x)&=&- \int_{M}  \triangle_i (E(y_1)  g_1(y_1+x_1,x^{0}))  dy_1
   \\
   &=&- \int_{M}  E(y_1) (\triangle_i g_1) (y_1+x_1,x^{0})  dy_1
   \\
   &=&- \int_{M}  E(y_1-x_1) (\triangle_i g_1) (y_1,x^{0})  dy_1.
   \end{eqnarray*}
Hence, under the assumption (iii) and then by Theorem \ref{Cauchy-Pompeiu},  it holds that
\begin{eqnarray*}
  \triangle_i  f(x)
  =- \int_{M}  E(y_1-x_1) (\overline{\partial}_1 (\partial_i g_i) )(y_1,x^{0})  dy_1= \partial_i g_i (x).
   \end{eqnarray*}
Consequently,  by assumption,
$$ \triangle_i(\overline{\partial}_j f- g_j )= \overline{\partial}_j \triangle_i f -\triangle_i   g_j=\overline{\partial}_j (\partial_i g_i )-\triangle_i   g_j=0, $$
which shows $\overline{\partial}_j f- g_j $  are harmonic in $M^{n}$ for  $j=1, \ldots,n$.  Note that, by construction,  $\overline{\partial}_j f- g_j $ are also compactly supported in $M^{n}$. Hence, the identity theorem for harmonic  functions can conclude that $\overline{\partial}_j f=g_j $ for $j=1, \ldots,n$. The proof is complete.
 \end{proof}
 \begin{remark}\label{C1-C00}{\rm
By Theorem \ref{Cauchy-Pompeiu-inverse}, the condition in (\ref{integral-condition}) has an equivalent form
$$\int_{M}\sum\limits_{s=0}^mv_s\big( E(y_1)(\partial_{x_{1,s}}g_j-\partial_{x_{j,s}}g_1)(y_1+x_1,x^{0})\big) dy_1=0,$$}
which appeared in the case of Clifford algebras \cite{Ren-Wang-14} and octonions \cite{Wang-Ren-14}. \end{remark}

As in the classical complex case,   we  can deduce a Hartogs extension theorem from Theorem \ref{inhomogeneous-theorem} for monogenic functions of   several hypercomplex variables over real alternative $\ast$-algebras, which  subsumes several results for  quaternions  (see \cite[Theorem 6]{Pertici-88} or \cite[Theorem l.1]{Adams-99}), Clifford algebras  (see \cite[Theorem 4.2]{Adams-99},  \cite[Theorem 4.3]{Ren-Wang-14}, and \cite[Proposition 8]{MM-04}), and  octonions \cite[Theorem 5.1]{Wang-Ren-14}.

\begin{theorem}[Hartogs]\label{Hartogs}
Let $\Omega$  be a domain  in $M^n$ with $n>1$. Let $K\subset\Omega$ be a compact set such that $\Omega \setminus K$ is connected.
Then every monogenic function   in $\Omega \setminus K$ can be extended to a monogenic   function on $\Omega$.
\end{theorem}

\begin{proof}
Let $f$ be a monogenic function in $\Omega\setminus K$. As in the classical complex case, pick  a function  $\varphi\in C_{0}^{\infty}(\Omega,\mathbb R)$ such that $\varphi\equiv1$ in a neighborhood   of $K$ and set $f_0=(1-\varphi)f$ in  $\Omega$.
We get readily  $f_0\in C^{\infty}(\Omega,\mathbb{A})\cap \mathcal {M}(\Omega\setminus\mathrm{supp}{\varphi}, \mathbb{A}).$
 For $j=1,2,\ldots,n$, consider
 \begin{eqnarray*}\label{eq:hp-U}
 h_j(x)=
 \left\{
 \begin{array}{lll}
 \overline{\partial}_j f_0(x), \quad x\in \ \Omega,
 \\
 0,\qquad\,\,\, x\in   M^n\setminus \Omega,
 \end{array}
 \right.
 \end{eqnarray*}
with $ h_j\in C_0^{\infty}(M^n, \mathbb{A})$  and $\mathrm{supp}{h_j}\subset\mathrm{supp}{\varphi}$.

By construction and   \cite[Proposition 5]{Perotti-22},  we have   in $\Omega$
$$  \overline{\partial}_i (\partial_j h_j) =\overline{\partial}_i ({\partial}_j(\overline{\partial}_j f_0))
 =\overline{\partial}_i (({\partial}_j\overline{\partial}_j ) f_0 )
=\overline{\partial}_i (\triangle_j f_0)= \triangle_j(\overline{\partial}_i f_0) =\triangle_j h_i,$$
i.e.
$$  \overline{\partial}_i (\partial_j h_j)=\triangle_j h_i,$$
which holds naturally outside $\Omega$.
Hence,   functions $h_j$ $(j=1,2,\ldots,n)$ satisfy the compatibility condition  (\ref{condition-Laplacde}).
  By Theorem \ref{inhomogeneous-theorem}, the inhomogeneous equation system
     $$ \overline{\partial}_j g=h_j,\quad j=1,2,\ldots,n,$$
has  a solution $g\in C_0^{\infty}(M^n, \mathbb{A})$, which vanishes on the unbounded  component of
    $M^n\setminus (\mathrm{supp}h_1\cup\cdots\cup\mathrm{supp}h_n)$.
Recalling that
    $$\mathrm{supp}h_j \subset \mathrm{supp}\varphi,\quad    j=1,\ldots,n,$$
particularly $g$ vanishes on the unbounded  component $\Omega_0$ of  $M^n\setminus \mathrm{supp}{\varphi} \subset  M^n\setminus (\mathrm{supp}h_1\cup\cdots\cup\mathrm{supp}h_n)$.

Set $\tilde{f}=f_0-g.$ We get  $\tilde{f}\in \mathcal {M}(\Omega, \mathbb{A})$, due to  that
   $$\overline{\partial}_j \tilde{f} = \overline{\partial}_j f_0-\overline{\partial}_j g=h_j-h_j=0.$$

In summary, we have
 $$ \left. g\right|_{\Omega_0}=0, \
                 \left.(f-f_0)\right|_{\Omega\setminus \mathrm{supp}{\varphi}}=0, \ \Omega_0\subset M^n\setminus \mathrm{supp}{\varphi},$$
which lead to  that  $\tilde{f}\equiv f$ on $\Omega_0\cap \Omega$. In view of that $ \Omega_0\cap \Omega\subset \Omega \setminus \mathrm{supp}{\varphi} \subset \Omega\setminus K$
        and    $\Omega\setminus K$ is connected, it follows from the identity principle  for harmonic functions that $\tilde{f}\equiv f$ on $\Omega\setminus K$, which  concludes that  $\tilde{f}$  is  the monogenic  extension of $f$ to $\Omega$, as desired.
\end{proof}
\par
\noindent
\textbf{Acknowledgements}
 The authors are grateful to the referee for careful reading of
the paper and valuable suggestions and comments, which improve the quality of this paper significantly.\\
\noindent
\textbf{Conflict of interest}
There is no financial or non-financial interests that are directly or indirectly related to the work submitted for publication.
\\
\textbf{Data availability}
Data sharing is not applicable to this article as no datasets were generated  during the current study.\\


\vskip 10mm

\begin{thebibliography}{99}
\bibitem{Uwe}R. Abreu Blaya,   J. Bory Reyes,  A. Guzm\'{a}n Ad\'{a}n, U. K\"{a}hler, \textit{On the $\varphi$-hyperderivative of the $\psi$-Cauchy-type integral in Clifford analysis}, Comput. Methods Funct. Theory 17 (2017), no. 1, 101-119.

\bibitem{Adams-99}
W. W. Adams, C. A. Berenstein, P. Loustaunau, I. Sabadini, D. C. Struppa,  \textit{Regular functions of several quaternionic variables and the Cauchy-Fueter complex}, J. Geom. Anal. 9 (1999), no. 1, 1-15.

\bibitem{Bochner-43}
S. Bochner, \textit{Analytic and meromorphic continuation by means of Green's formula},
Ann. of Math. (2) 44 (1943), 652-673.



\bibitem{Bory-18}
J. Bory Reyes,   A. Guzm\'{a}n Ad\'{a}n,   F. Sommen,
\textit{Bochner-Martinelli formula in superspace},  Math. Methods Appl. Sci. 41 (2018), no. 18, 9449-9476.
\bibitem{BDDS-09}
 F. Brackx, B. De Knock, H. De Schepper, F. Sommen, \textit{On Cauchy and Martinelli-Bochner integral formulae in Hermitean Clifford analysis}, Bull. Braz. Math. Soc. (N.S.) 40 (2009), no. 3, 395-416.
\bibitem{Brackx} F. Brackx, R. Delanghe, F. Sommen, \textit{Clifford analysis}, Research Notes in Mathematics, Vol. 76, Pitman, Boston, 1982.

\bibitem{Brackx-Pincket-84}
F. Brackx, W. Pincket,
\textit{A Bochner-Martinelli formula for the biregular functions of Clifford analysis},
Complex Variables Theory Appl. 4 (1984), no. 1, 39-48.



  \bibitem{Colombo-Sabadini-Sommen-Struppa-04}
  F. Colombo, I. Sabadini, F. Sommen,  D. C.  Struppa,    \textit{Analysis of Dirac systems and computational algebra}, Progress in Mathematical Physics, Vol. 39, Birkh\"auser Boston,  2004.


\bibitem{Delanghe-Sommen-Soucek-92}R. Delanghe, F. Sommen, V. Sou\v{c}ek, \textit{Clifford algebra and spinor-valued functions. A function theory for the Dirac operator}, Mathematics and its Applications, Vol. 53, Kluwer Academic Publishers Group, Dordrecht, 1992.
 \bibitem{Dentoni-Sce} P. Dentoni, M. Sce, \textit{Funzioni regolari nell'algebra di Cayley},   Rend. Sem. Mat. Univ. Padova 50 (1973), 251-267.
   \bibitem{Ghiloni-Perotti} R.  Ghiloni,  A. Perotti, \textit{Slice regular functions on real alternative algebras}, Adv. Math. 226 (2011), no. 2, 1662-1691.
   \bibitem{Ghiloni-Stoppato-24-2} R. Ghiloni, C. Stoppato, \textit{A unified theory of regular functions of a hypercomplex variable}, Bull. Sci. Math. 209 (2026), Paper No. 103794, 92 pp.
\bibitem{Gurlebeck} K. G\"{u}rlebeck, K. Habetha, W. Spr\"{o}{\ss}ig, \textit{Holomorphic functions in the plane and $n$-dimensional space}, Birkh\"{a}user Verlag, Basel, 2008.

 \bibitem{HRX} Q. Huo, G. Ren, Z. Xu,
\textit{Monogenic functions over real alternative $\ast$-algebras: fundamental results and applications}, arXiv:2504.01359, 2025.

 \bibitem{Kytmanov-95}
A. M. Kytmanov,
\textit{The Bochner-Martinelli integral and its applications},  Birkh\"auser Verlag, Basel, 1995.

\bibitem{Li-Peng-02}X. Li,  L. Peng, \textit{The Cauchy integral formulas on the octonions}, Bull. Belg. Math. Soc. Simon Stevin 9 (2002), no. 1, 47-64.
\bibitem{Li-Ren-22}
Y. Li, G.  Ren,  \textit{ Monogenic Cauchy implies holomorphic Bochner-Martinelli}, Adv. Appl. Clifford Algebr. 32 (2022), no. 3, Paper No. 34, 19 pp.


\bibitem{Lin-88}
L. Lin,  \textit{The boundary behavior of Cauchy type on a closed pieeewise smooth manifold},  
 Acta Math. Sinica, 31 (1988), 547-557.
\bibitem{Lin-Qiu-88}
L. Lin, C. Qiu, \textit{The singular integral equation on a closed piecewise smooth manifold in $\mathbb{C}^{n}$}, Integral Equations Operator Theory 44 (2002), no. 3, 337-358.
\bibitem{Lin-Qiu-04}
L. Lin, C. Qiu, \textit{The Cauchy boundary value problems on closed piecewise smooth manifolds in $\mathbb{C}^{n}$}, Acta Math. Sin. (Engl. Ser.) 20 (2004), no. 6, 989-998.
\bibitem{Lu-Zhong-57}
Q.  Lu, T.   Zhong,  \textit{ An extension of the Privalov theorem},  Acta Math. Sinica,
 7 (1957), 144-165.

\bibitem{MM-04}
E. Marmolejo-Olea, M. Mitrea,  \textit{Harmonic analysis for general first order differential operators in Lipschitz domains}, Clifford algebras, 91-114, Prog. Math. Phys., 34, Birkh\"auser Boston, Boston, MA, 2004.


\bibitem{Martinelli-38}
E. Martinelli,  \textit{Alcuni teoremi integrali per le funzioni analitiche di pi\`{u} variabili complesse}. Mem. della
 Reale Accad. d'Italia 9, (1938), 269-283.
  \bibitem{Okubo} S. Okubo, \textit{
Introduction to octonion and other non-associative algebras in physics},
Montroll Memorial Lecture Series in Mathematical Physics, 2. Cambridge University Press, Cambridge, 1995.
\bibitem{Perotti-22} A. Perotti, \textit{Cauchy-Riemann operators and local slice analysis over real alternative algebras}, J. Math. Anal. Appl. 516 (2022), no. 1, Paper No. 126480, 34 pp.


\bibitem{Pertici-88}
D. Pertici, \textit{Regular functions of several quaternionic variables},  Ann. Mat. Pura Appl. (4) 151 (1988), 39-65.
 \bibitem{Qian-Zh-01}
 T. Qian,  T. Zhong,  \textit{The differential integral equations on smooth closed orientable manifolds}, Acta Math. Sci. Ser. B (Engl. Ed.) 21 (2001), no. 1, 1-8.

 \bibitem{Range-86}
R. M. Range, \textit{Holomorphic functions and integral representations in several complex variables}, Graduate Texts in Mathematics, 108. Springer-Verlag, New York, 1986.

  \bibitem{Ren-Shi-Wang-20}G. Ren, Y. Shi, W. Wang,  \textit{The tangential $k$-Cauchy-Fueter operator and $k$-CF functions over the Heisenberg group}, Adv. Appl. Clifford Algebr. 30 (2020), no. 2, Paper No. 20, 19 pp.


 \bibitem{Ren-Wang-14}
 G.  Ren,  H. Wang, \textit{Theory of several paravector variables: Bochner-Martinelli formula and Hartogs theorem}, Sci. China Math. 57 (2014), no. 11, 2347-2360.
  \bibitem{RSS-02}
R. Rocha-Ch\'{a}vez, M. Shapiro, F. Sommen,  \textit{Integral theorems for functions and differential forms in ${\bf C}^m$},
 Chapman and Hall/CRC, Boca Raton, FL, 2002.


 \bibitem{Schafer} R. D. Schafer, \textit{An introduction to nonassociative algebras}, Pure and Applied Mathematics, Vol. 22. Academic Press,
 New York-London, 1966.

 \bibitem{Sommen-87}
F.  Sommen,
\textit{Martinelli-Bochner type formulae in complex Clifford analysis},
Z. Anal. Anwendungen 6 (1987), no. 1, 75-82.



\bibitem{Wang-Ren-14}
H. Wang, G.  Ren,  \textit{Octonion analysis of several variables}, Commun. Math. Stat. 2 (2014), no. 2, 163-185.
\bibitem{Wang-Ren-14-JMP}
H. Wang, G.  Ren,  \textit{Bochner-Martinelli formula for $k$-Cauchy-Fueter operator}, J. Geom. Phys. 84 (2014), 43-54.

\bibitem{Wang-10}
W. Wang,   \textit{The $k$-Cauchy-Fueter complex, Penrose transformation and Hartogs phenomenon for quaternionic $k$-regular functions}, J. Geom. Phys. 60 (2010), no. 3, 513-530.

\bibitem{Xu-Sabadini-25-O} Z. Xu, I. Sabadini, \textit{Generalized partial-slice monogenic functions: the octonionic case},  Trans. Amer. Math. Soc. DOI: 10.1090/tran/9646, 2026, online.



\end{thebibliography}
\end{document}